\documentclass[12pt]{article}
\usepackage[utf8]{inputenc}
\usepackage{tikz}
\usetikzlibrary {positioning}

\title{Hat Guessing on Books and Windmills}
\author{Xiaoyu He\thanks{Dept. of Mathematics, Stanford University, {\tt alkjash@stanford.edu}. Research supported by NSF Graduate Research Fellowship Grant No. DGE-1656518.} \and Yuzu Ido\thanks{Stanford University, {\tt yuzu@stanford.edu}.} \and Benjamin Przybocki\thanks{Stanford University, {\tt benprz@stanford.edu}.}}
\date{October 2020}

\usepackage{graphicx}
\usepackage{amsmath,amsthm,amssymb,mathabx,mathrsfs}
\usepackage{fullpage}
\usepackage{mathtools}
\usepackage{hyperref}
\usepackage{comment}
\usepackage[english]{babel}
\usepackage{calrsfs}

\hypersetup{
colorlinks = true,
citecolor = blue,
linkcolor = black
}

\newcommand{\Z}{\mathbb{Z}}
\newcommand{\N}{\mathbb{N}}
\newcommand{\HG}{\mathrm{HG}}
\newcommand{\bHb}{\text{B}}

\newcommand{\digit}[3]{\text{d}_{#1} ({#2},{#3})}
\newcommand{\ruler}[2]{\nu_{#1} ({#2})}

\DeclarePairedDelimiter\floor{\lfloor}{\rfloor}

\theoremstyle{definition}
\newtheorem{definition}{Definition}[section]

\theoremstyle{plain}
\newtheorem{theorem}[definition]{Theorem}

\newtheorem{question}[definition]{Question}

\newtheorem{lemma}[definition]{Lemma}

\newtheorem{conjecture}[definition]{Conjecture}

\theoremstyle{remark}
\newtheorem*{claim}{Claim}

\begin{document}

\maketitle

\begin{abstract}
    The hat-guessing number is a graph invariant defined by Butler, Hajiaghayi, Kleinberg, and Leighton. We determine the hat-guessing number exactly for book graphs with sufficiently many pages, improving previously known lower bounds of He and Li and exactly matching an upper bound of Gadouleau. We prove that the hat-guessing number of $K_{3,3}$ is $3$, making this the first complete bipartite graph $K_{n,n}$ for which the hat-guessing number is known to be smaller than the upper bound of $n+1$ of Gadouleau and Georgiou. Finally, we determine the hat-guessing number of windmill graphs for most choices of parameters.
\end{abstract}

\section{Introduction} \label{sec-intro}
Hat-guessing games are combinatorial games in which players try to guess the colors of their own hats. In the variant we study, defined by Butler, Hajiaghayi, Kleinberg, and Leighton~\cite{butler2008hat}, each player is assigned 1 of $q$ possible hat colors and is placed at a vertex of a graph $G$. Players can see the hat colors of the players at adjacent vertices, but not their own. Players can communicate to design a collective strategy before the hats are assigned by the adversary. Once hats are assigned, the players must simultaneously guess the colors of their own hats, and they collectively \textit{win} if at least one player guesses correctly.
\begin{definition} \label{def-HG}
    The \emph{hat-guessing number} of the graph $G$, denoted $\HG(G)$, is the largest number of hat colors $q$ for which the players can guarantee a win in the hat-guessing game on $G$.
\end{definition}
This version of the hat-guessing game has found connections to derandomizing auctions~\cite{aggarwal,benzwi} and recently to coding theory and finite dynamical systems~\cite{gadouleau}.

The most famous special case of the hat-guessing game, where $G=K_n$ is the complete graph, was popularized by Winkler~\cite{winkler} in one of his beautiful puzzle collections. Here, $n$ players can all see each other, and the game is to show that $\HG\left(K_n\right) = n$. The strategy which wins on $n$ colors is as follows: identify the hat colors with the set\footnote{Hereafter, $[q]$ denotes the set $\{0, 1, \dots, q-1\}$.} $[n]$. Player $i$ guesses the hat color that would make the sum of all the hat colors $i \pmod n$. Since the actual sum of everyone's hat colors must take some value in $\mathbb{Z}/n\mathbb{Z}$, exactly one player will guess correctly. Conversely, it is not difficult to show that the players cannot guarantee a win when $n+1$ colors are available for the adversary.

The hat-guessing numbers of graphs other than the complete graph have proven surprisingly difficult to compute. The value of $\HG(G)$ has been determined for trees~\cite{butler2008hat}, cycles~\cite{szczechla}, extremely unbalanced complete bipartite graphs~\cite{alon2018hat}, and certain tree-like degenerate graphs~\cite{he2020hat}, but outside of these very specific families little is known. In this paper we add to this list of solved graphs almost all books and windmills, as well as the graph $K_{3,3}$, which is in some sense the first ``interesting'' complete bipartite graph for this problem.

The {\it book graph} $B_{d, n}$ is obtained by adding $n$ nonadjacent common neighbors to the complete graph $K_d$. The $d$-clique is called the \emph{spine} of $B_{d, n}$ and the other $n$ vertices are called its \emph{pages}. Book graphs were originally studied by Bosek, Dudek, Farnik, Grytczuk, and Mazur~\cite{bosek2019hat} in this context. They are examples of $d$-degenerate graphs for which $\HG(G)$ can be exponentially large in $d$. Independently, Gadouleau~\cite[Theorem~3]{gadouleau} proved a general upper bound that implies
\begin{equation}\label{eq:gadouleau}
\HG(G) \leq 1 + \sum_{i = 1}^{\tau(G)} i^i, 
\end{equation}
where $\tau(G)$ is the size the minimum vertex cover of $G$. As $B_{d,n}$ is the unique maximal graph on $n$ vertices with $\tau(G)=d$, determining $\HG(B_{d,n})$ is actually equivalent to finding the best possible upper bound on $\HG(G)$ in terms of $\tau(G)$. Our first main result is that Gadouleau's upper bound~(\ref{eq:gadouleau}) is tight for books, and thus best possible.
\begin{theorem} \label{thm-book}
    For $d \ge 1$ and $n$ sufficiently large in terms of $d$, $\HG\left(B_{d, n}\right) = 1 + \sum_{i = 1}^{d} i^i$.
\end{theorem}
It was shown by~\cite{bosek2019hat} that $\HG(B_{d,n}) \ge 2^d$ for sufficiently large $n$ in terms of $d$ by reducing upper bounds on $\HG(B_{d,n})$ to a certain geometric problem about counting projections in $\mathbb{N}^d$. He and Li~\cite{he2020hat} showed that this geometric problem is actually equivalent to determining $\HG(B_{d,n})$ for $n$ sufficiently large and improved the lower bound to $\HG\left(B_{d, n}\right) \geq (d+1)!$. Our proof of Theorem~\ref{thm-book} solves the equivalent geometric problem completely using Hall's Marriage Theorem. 

Perhaps the most well-studied case of the hat-guessing game is the complete bipartite case. In the paper defining the hat-guessing game~\cite{butler2008hat}, it was proved that for large $n$, $\HG\left(K_{n,n}\right) = \Omega(\log \log n)$. Later, Gadouleau and Georgiou~\cite{gadgeor} proved that $\Omega(\log n) \le \HG\left(K_{n,n}\right) \le n+1$, and most recently, Alon, Ben-Eliezer, Shangguan, and Tamo~\cite{alon2018hat} improved the lower bound to $\HG\left(K_{n,n}\right) = \Omega(n^{\frac{1}{2} - o(1)})$. However, the exact value of $\HG\left(K_{n,n}\right)$ was only known in the cases $n=1,2$. Our next result solves the problem for $n=3$. 
\begin{theorem} \label{thm-k33}
    For the complete bipartite graph $K_{3,3}$, we have $\HG\left(K_{3,3}\right) = 3$.
\end{theorem}
This is the first example where the upper bound $\HG\left(K_{n,n}\right) \le n+1$ of~\cite{gadgeor} is known not to be tight, and suggests that $\HG\left(K_{n,n}\right)$ may be smaller than linear in general.

Finally, we consider the hat-guessing numbers of windmill graphs $W_{k,n}$, defined as $n$ disjoint copies of $K_k$ glued together at a single vertex. Thus $W_{k,n}$ has a total of $(k-1)n+1$ vertices. One might initially suspect that $\HG(W_{k,n})$ cannot be much larger than $k$, since except for the central vertex, $W_{k,n}$ consists of $n$ disjoint copies of $K_{k-1}$. We show to the contrary that $\HG(W_{k,n})$ can be almost twice as large as $k$ in general.
\begin{theorem} \label{thm-wd-2k-2}
    For $k \ge 2$ and $n \ge \log_2(2k-2)$, $\HG(W_{k,n}) = 2k-2$.
\end{theorem}
Theorem~\ref{thm-wd-2k-2} determines $\HG(W_{k,n})$ when $n$ is sufficiently large. Similar methods work for smaller choices of $n$.
\begin{theorem} \label{thm-wd-dn}
    For any $n\ge 1$ and $d\ge 2$, we have $\HG(W_{d^n - d^{n-1} + 1,n}) = d^n$.
\end{theorem}
In fact, it is not difficult to generalize this construction and show that $\HG(W_{k,n})\approx k+k^{1-1/n}$ in general.

In Section~\ref{sec-books}, we study book graphs and prove Theorem~\ref{thm-book} by solving the equivalent geometric problem. Then, in Section~\ref{sec-k33}, we prove Theorem~\ref{thm-k33} by reducing it to certain partitioning and covering problems in a cube. In Section~\ref{sec-windmills}, we study windmill graphs and prove Theorems~\ref{thm-wd-2k-2} and \ref{thm-wd-dn}. Finally, in Section~\ref{sec-concl}, we present a few of the many attractive open problems in this area.

\section{Books} \label{sec-books}

Recall that the {\it book graph} $B_{d, n}$ is obtained by adding $n$ nonadjacent common neighbors to the complete graph $K_d$. The $d$-clique is called the \emph{spine} of $B_{d, n}$ and the other $n$ vertices are called its \emph{pages}. The hat-guessing number of books can was reduced by~\cite{bosek2019hat} and~\cite{he2020hat} to a geometric problem.
\begin{definition}
    A set $S \subset \N^d$ is \emph{coverable} if there is a partition $S = S_1 \sqcup S_2 \sqcup \cdots \sqcup S_d$ such that $S_i$ contains at most one point along any line parallel to the $i$-th coordinate axis.
\end{definition}
For example, the set $[2]^2$ is coverable in $\mathbb{N}^2$ because it has the partition $S_1 = \{(0,0), (1,1)\}$, $S_2 = \{(0,1),(1,0)\}$ so that $S_i$ has at most one point along any axis-parallel line, but the set $[2]\times [3]$ has no such partition and is not coverable.

Let $h(\mathbb{N}^d)$ be the largest $t$ such that every $t$-subset of $\mathbb{N}^d$ is coverable. It was shown by~\cite{he2020hat} that for $n$ sufficiently large in terms of $d$, $\HG(B_{d,n})=h(\mathbb{N}^d) + 1$. In other words $\HG(B_{d,n})$ is the size of the smallest non-coverable set in $d$ dimensions. Below we compute $h(\mathbb{N}^d)$ exactly by reformulating coverability as a matching condition and applying Hall's Marriage Theorem. The corresponding neighborhood condition is as follows.
\begin{definition}
    A set $S \subset \N^d$ is \emph{numerically coverable} if $\sum_{i=1}^d |\pi_i(S)| \ge |S|$, where $\pi_i(S)$ is the $(d-1)$-dimensional projection of $S$ onto the $i$-th coordinate hyperplane.
\end{definition}
The following key lemma reduces checking coverability to checking numerical coverability.
\begin{lemma}\label{lemma-cov-iff-numcov}
     A set $S \subset \N^d$ is coverable if and only if every subset of $S$ is numerically coverable.
    \end{lemma}
    \begin{proof}
    Suppose first that a set $S$ is coverable by the partition $S_1 \sqcup S_2 \sqcup \cdots \sqcup S_d$. For any subset $T \subseteq S$, let $T_i= S_i \cap T$, so that $T = T_1 \sqcup T_2 \sqcup \cdots \sqcup T_d$ is a partition of $T$ which certifies the coverability of $T$. By the definition of coverability, we get $|\pi_i(T)| \geq |T_i|$ for all $i$, so $\sum_{i=1}^d |\pi_i(T)| \ge \sum_{i=1}^d |T_i| = |T|$. Thus every subset of $S$ is numerically coverable.
    
    Now suppose every subset of $S$ is numerically coverable. We use the asymmetric version of Hall's Marriage Theorem~\cite{hall}, which states that a bipartite graph $G$ on sets $U$ and $V$ contains a perfect matching from $U$ to $V$ if every subset $U'$ of $U$ has at least $|U'|$ total neighbors in $V$. 
    
    We define a bipartite graph $G$ to apply Hall's Theorem to as follows. The left side is our set $S$, and the right side is the set $L$ of axis-parallel lines intersecting $S$. An edge $(s,\ell)\in S\times L$ lies in $G$ if and only if line $\ell$ contains point $s$. If every subset of $S$ is numerically coverable, this means that every set $T\subseteq S$ of points lies on at least $|T|$ distinct axis-parallel lines in $L$, and so the conditions of Hall's Marriage Theorem are satisfied. Thus, a perfect matching from $S$ to $L$ exists.
    
    Given a perfect matching from $S$ to $L$, we can construct a partition for $S$ exhibiting its coverability. Indeed, let $S_i$ be the subset of $S$ matched to lines in $L$ parallel to the $x_i$-axis. This is a partition of $S$ with the property that $S_i$ contains at most one point along each $x_i$-axis, and so $S$ is coverable, as desired.
    \end{proof}
    
    Our goal in the next two lemmas is to show that all small enough sets are numerically coverable. This can be done entirely by applying results of Lev and Rudnev~\cite{lev} determining the sets minimizing $\sum_{i=1}^d |\pi_i(S)|$ for any given fixed size $|S|$. However, for simplicity of exposition we break it into two parts.

    \begin{lemma}\label{lemma-dTod-coverable}
    Any set $S\subset \N^d$ of size at most $d^d$ is numerically coverable.
    \end{lemma}
    \begin{proof}
    We use a special case of the Loomis-Whitney inequality~\cite{loomis}:
    \[
    |S|^{d-1} \le \prod_{i=1}^d |\pi_i(S)|.
    \]
    Applying the arithmetic mean\textendash geometric mean inequality to the inequality above implies
    \begin{align*}
    |S|^{d-1} &\le \left(\frac{1}{d} \sum_{i=1}^d |\pi_i(S)|\right)^d \\
    d|S|^{\frac{d-1}{d}} &\le \sum_{i=1}^d |\pi_i(S)| \label{eq: 2}\tag{2}.
    \end{align*}
    Since $|S| \le d^d$, we have $|S| \le d|S|^{\frac{d-1}{d}}$, and so~(\ref{eq: 2}) gives $|S| \le \sum_{i=1}^d |\pi_i(S)|$ as desired. 
    \end{proof}
    It remains to show that if $|S|$ is between $d^d$ and $\sum_{i=1}^d i^i$, $S$ is still numerically coverable.
    \begin{lemma}\label{lemma-sum-numcov}
    Any set $S\subset \N^d$ of size at most $\sum_{i = 1}^d i^i$ is numerically coverable.
    \end{lemma}
    \begin{proof}
    We already know the claim holds for $|S| \leq d^d$ by Lemma~\ref{lemma-dTod-coverable}. Thus, we may assume $S$ satisfies $d^d < |S| \leq \sum_{i = 1}^d i^i$.
    
    Assume for the sake of contradiction that the lemma does not hold for all $d \in \mathbb{N}$. Then, there must be some minimum $d$ for which it does not hold. Pick this smallest $d$ for which the lemma is false and let $S$ be a minimum counterexample in this dimension $d$; that is, $d^d < |S| \leq \sum_{i = 1}^d i^i$ and $|S| > \sum_{i=1}^d |\pi_i(S)|$. We may further assume that $S$ minimizes $\sum_{i=1}^d |\pi_i(S)|$ among all sets of the same size $|S|$. 
    
    Lev and Rudnev~\cite{lev} determined the exact sets $S$ of any fixed size minimizing $\sum_{i=1}^d |\pi_i(S)|$. It is a straightforward deduction from their results that we may assume that an optimal $S$ contains the hypercube $[d]^d$ and that $S \setminus [d]^d$ lies in one hyperface adjacent to the hypercube. Without loss of generality, say the hyperface in question is the one with $x_d = d+1$. In other words, we assume
    \[
    [d]^d \subseteq S \subseteq [d]^{d-1}\times [d+1]. 
    \]
    This implies that $\pi_i(S) = \pi_i\left([d]^d \right) + \pi_i\left(S \setminus [d]^d\right)$ for $i= 1,\ldots, d-1$ and that $\pi_d(S) = \pi_d\left([d]^d \right)$.
    
    Then,
    \begin{align*}
        \lvert S \rvert &> \sum_{i=1}^d \lvert \pi_i(S) \rvert\\
        \lvert S \rvert - \lvert [d]^d \rvert = \lvert S \setminus [d]^d \rvert &>  \sum_{i=1}^d \lvert \pi_i(S) \rvert -  \sum_{i=1}^d \lvert \pi_i\left([d]^d \right) \rvert\\
        \lvert S \setminus [d]^d \rvert &> \sum_{i=1}^{d-1} \pi_i\left(S \setminus [d]^d\right) 
    \end{align*}
    Since $\lvert S \rvert \leq \sum_{i = 1}^d i^i$, we see $\lvert S \setminus [d]^d \rvert \leq \sum_{i = 1}^{d-1} i^i$. Thus $\pi_d(S \setminus [d]^d)$ is a counterexample to the claim in dimension $d-1$, contradicting the minimality of $d$. 
    \end{proof}
Now Theorem~\ref{thm-book} follows immediately, since Lemmas~\ref{lemma-cov-iff-numcov} and \ref{lemma-sum-numcov} together prove that $h(\N^d) \geq \sum_{i = 1}^d i^i$, and the matching upper bound was shown by Gadouleau~\cite{gadouleau}. For completeness, we include a quick sketch of this upper bound construction.

\begin{lemma} \label{lem-bk-upper}
    For $d\ge 1$, we have $h(\N^d) \leq \sum_{i = 1}^d i^i.$
\end{lemma}
\begin{figure}[h!]
    \centering
    \includegraphics[scale=0.75]{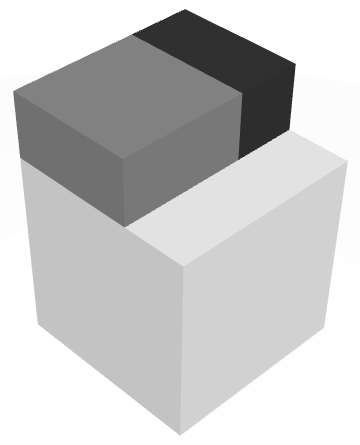}
    \caption{A non-coverable 33-set in $\N^3$}
    \label{fig:uncoverable}
\end{figure}
\begin{proof}[Proof Sketch.]
We construct a non-coverable set of size $1+ \sum_{i=1}^d i^i$ by induction. First, it is evident that $h(\N^1)=1$. Next, given a set $S \subset \N^d$ of size $1+ \sum_{i=1}^d i^i$ that is non-coverable, we can create a non-coverable set of size $1+ \sum_{i=1}^{d+1} i^i$ in $\N^{d+1}$ as follows. Simply take a $[d+1]^{d+1}$ hypercube and position a copy of $S$ inside a coordinate hyperplane adjacent to one of its faces. In Figure~\ref{fig:uncoverable}, which shows the case $d = 3$, the white region is the $3 \times 3 \times 3$ cube, and the gray and black regions are $S$.
\end{proof}

\section{The Complete Bipartite Graph $K_{3,3}$} \label{sec-k33}
The hat-guessing number of complete bipartite graphs relates closely to packing combinatorial cubes, as defined below.
\begin{definition}
    In three dimensions, an $l \times m \times n$ \emph{combinatorial prism} is a Cartesian product of one $l$-set, one $m$-set, and one $n$-set. If $l = m = n$, it is called a \emph{combinatorial cube}. ``Combinatorial prisms'' and ``combinatorial cubes'' will be abbreviated as ``prisms'' and ``cubes'' respectively.
\end{definition}
The following lemma explicitly states the relation between cubes and hat-guessing on complete bipartite graphs. It is a specific case of machinery for complete bipartite graphs presented in~\cite{alon2018hat}.
\begin{lemma}\label{lemma-cube-condition}
    We have $\HG\left(K_{3,3}\right) \geq 4$ if and only if there exist three partitions $P$, $Q$, and $R$ given by \[[4]^3 = P_1 \sqcup P_2 \sqcup P_3 \sqcup P_4 = Q_1 \sqcup Q_2 \sqcup Q_3 \sqcup Q_4 = R_1 \sqcup R_2 \sqcup R_3 \sqcup R_4\] such that $P_i \cup Q_j \cup R_k$ contains a $3 \times 3 \times 3$ cube for all choices of $1\le i, j, k \le 4$.
\end{lemma}
\begin{proof}
We show the proof of the if direction. In $K_{3,3}$, call the left and right parts $V_L$ and $V_R = \{p, q, r\}$ respectively. Define $v_L\in [4]^3$ to be the vector of hat color assignments on $V_L$. Then, the guessing strategies of vertices $p$, $q$, and $r$ will be built from the partitions $P$, $Q$, and $R$, respectively. Specifically, let $p$ guess color $i-1$ exactly when $v_L \in P_i$. Similarly define the hat-guessing strategies on $q$ and $r$ in terms of the partitions $Q$ and $R$ respectively.

It remains to give the guessing strategy on the left hand side. Since we only need one vertex total to guess correctly, the vertices in $V_L$ may assume that each of $p$, $q$, and $r$ guesses incorrectly. If the vertices of $V_L$ see colors $i$, $j$, and $k$ on vertices $p$, $q$, and $r$ respectively, this assumption implies that $v_L \not \in P_i \cup Q_j \cup R_k$. Recalling that every such union contains a $3\times 3\times 3$ cube, we see that $v_L$ must be in the complement of such a cube. But the complement of a $3\times 3\times 3$ cube $C$ in $[4]^3$ is the Hamming ball of radius $2$ about some point $(x,y,z)\in [4]^3$, i.e. every point outside $C$ shares a coordinate with $(x,y,z)$. Thus, if the three vertices on the left guess colors $x$, $y$, and $z$, respectively, at least one of them guesses correctly, as desired.

The only if direction is similar. Given a winning guessing strategy, the partitions $P$, $Q$, and $R$ are exactly those given by the fibers of the guessing functions of the right hand side vertices $p$, $q$, and $r$.
\end{proof}
It will be convenient to study two-fold, and not mutual, intersections of set families.
\begin{definition} \label{def-m-inter}
    A point $p$ is a \emph{two-intersection point} of a family of sets $\{S_1, \dots, S_n\}$ ($n\ge 2$) if it is contained in at least two distinct sets $S_i$ and $S_j$. The set of all two-intersection points of $\{S_1, \dots, S_n\}$ is simply called the \emph{two-intersection} of this family.
\end{definition}
Before we complete the proof of Theorem~\ref{thm-k33}, we will need three technical lemmas about the intersection patterns of cubes and prisms in $[4]^3$.
\begin{lemma} \label{lemma-cube-or-minusPoint}
    If four $3 \times 3 \times 3$ cubes in $[4]^3$ two-intersect in at most 29 points, then their two-intersection is either a $3 \times 3 \times 3$ cube or a $3 \times 3 \times 3$ cube missing one point.
\end{lemma}
\begin{lemma} \label{lemma-three-cubes}
Three $3 \times 3 \times 3$ cubes in $[4]^3$ must two-intersect in at least 20 points. 
\end{lemma}
The preceding two lemmas will be proved in the appendix with finite case checks.

\begin{lemma}\label{lemma-four-prisms-impossible}
It is impossible for four $16$-sets, each lying inside some $2 \times 3 \times 3$ prism, to partition $[4]^3$.
\end{lemma}
\begin{proof}
We claim that three $3 \times 3 \times 3$ cubes and one $2 \times 3 \times 3$ prism cannot cover $[4]^3$, from which it follows that four $2 \times 3 \times 3$ prisms cannot cover $[4]^3$. This would certainly be enough. 

Assume for the sake of contradiction that such a covering were possible and look at the $x$-coordinates missing from each set. Without loss of generality, the $2\times 3 \times 3$ prism is oriented so that two $x$-coordinates are missing. One possible $x$-coordinate is missing from each from the cubes, for a total of five, so some $x$-coordinate is missing twice. Therefore in that $4 \times 4$ cross section of $[4]^3$, only two of the four sets appear. However, it is impossible to cover a $4 \times 4$ square with at most two $3 \times 3$ squares, so we have arrived at the desired contradiction.
\end{proof}

We are now ready to prove Theorem~\ref{thm-k33}, which states that $\HG\left(K_{3,3}\right) = 3$.

\begin{proof}[Proof of Theorem~\ref{thm-k33}]

Since $\HG\left(K_{2,2}\right) = 3$ and it is a subgraph of $K_{3,3}$, $\HG\left(K_{3,3}\right) \geq 3$. Since $\HG\left(K_{m,n}\right) \leq \min(m, n) + 1$ by a result of~\cite{gadgeor}, we know that $\HG\left(K_{3,3}\right) \leq 4$. It remains to show that $\HG\left(K_{3,3}\right) \neq 4$.

Suppose for the sake of contradiction that $\HG\left(K_{3,3}\right) = 4$.  By Lemma~\ref{lemma-cube-condition}, $\HG\left(K_{3,3}\right) = 4$ if and only if there are three partitions of a $4 \times 4 \times 4$ cube into four parts each, such that the union of one part from each partition always contains a $3 \times 3 \times 3$ cube (i.e., a set of the form\footnote{The complement of a set $S$ is denoted by $\overline{S}$.} $\overline{\{p\}} \times \overline{\{q\}} \times \overline{\{r\}}$, for $p, q, r \in [4]$).

We will denote the three partitions of $[4]^3$ as $P$, $Q$, and $R$. Without loss of generality, the parts in $P$, which are $P_1$, $P_2$, $P_3$, and $P_4$, are labeled such that $|P_1| \le |P_2| \le |P_3| \le |P_4|$. The parts in $Q$, which are $Q_1$, $Q_2$, $Q_3$, and $Q_4$, are labeled such that $|Q_1 \setminus P_1| \le |Q_2 \setminus P_1| \le |Q_3 \setminus P_1| \le |Q_4 \setminus P_1|$. The parts in $R$, which are $R_1$, $R_2$, $R_3$, and $R_4$, are labeled arbitrarily. 

The first part of this proof is to show that partition $P$ must be balanced in order for all choices of $P_i \cup Q_j \cup R_k$ to contain a $3\times 3 \times 3$ cube. The idea is to repeatedly exploit the fact that since the $R_k$ are disjoint sets, $P_i \cup Q_j$ contains the two-intersection of four $3\times 3\times 3$ cubes in $[4]^3$.

By these assumptions, we get $|P_1| \leq 16$ and $|Q_1 \setminus P_1| \le 12$, so $|P_1 \cup Q_1| \leq 28$. By Lemma~\ref{lemma-cube-condition}, $P_1 \cup Q_1 \cup R_i$ must contain a $3 \times 3 \times 3$ cube $D_i$ for $1 \leq i \leq 4$. Since the $R_i$ are disjoint, $\bigcap_{i = 1}^4 (P_1 \cup Q_1 \cup R_i) = P_1 \cup Q_1$. In fact, any point in the two-intersection of $D_1, \dots, D_4$ must lie inside this set $P_1 \cup Q_1$ of at most 28 points (see Definition~\ref{def-m-inter}). By Lemma~\ref{lemma-cube-or-minusPoint}, if $D_1, \dots, D_4$ two-intersect in at most 29 points, then their two-intersection is either a $3 \times 3 \times 3$ cube or a $3 \times 3 \times 3$ cube missing one point. It follows that $|P_1 \cup Q_1| \geq 26$, so $|Q_1 \setminus P_1| \geq 10$. We can now apply the pigeonhole principle to find that $|Q_2 \setminus P_1|  \leq \floor*{\frac{1}{3} (4^3 - 26)} = 12$. We consider whether $|Q_2 \setminus P_1| \geq 11$ or not.
\begin{claim}
$|Q_2 \setminus P_1| \le 10$.
\end{claim}
\begin{proof}
Assume for the sake of contradiction that $|Q_2 \setminus P_1| \geq 11$. Then, we can apply the pigeonhole principle again to find that $|Q_3 \setminus P_1| \leq \floor*{\frac{1}{2} (4^3 - 26 - 11)} = 13$. Thus, $|Q_1 \setminus P_1| \leq |Q_3 \setminus P_1| \leq 13$ and $|Q_2 \setminus P_1| \leq |Q_3 \setminus P_1| \leq 13$. Recall that $|P_1| \leq 16$. Then, $|P_1 \cup Q_1|$, $|P_1 \cup Q_2|$, and $|P_1 \cup Q_3|$ are all at most 29. Furthermore by applying Lemma~\ref{lemma-cube-condition} and Lemma~\ref{lemma-cube-or-minusPoint}, the sets $P_1 \cup Q_1$, $P_1 \cup Q_2$, and $P_1 \cup Q_3$ must each contain all but at most one point of some $3 \times 3 \times 3$ cube. 

Since the $Q_i$ are disjoint, $\bigcap_{i = 1}^3 (P_1 \cup Q_i) = P_1$. These three sets, each missing at most one point from some $3 \times 3 \times 3$ cube, two-intersect in at most $|P_1|\le 16$ points. However, by Lemma~\ref{lemma-three-cubes}, three full $3 \times 3 \times 3$ cubes must two-intersect in at least $20$ points. Removing one point from a $3 \times 3 \times 3$ cube can remove at most one point from the resulting two-intersection. Thus the two-intersection of $\{P_1 \cup Q_i\}_{i=1}^3$ must contain at least $20-3=17$ points. This is a contradiction.
\end{proof}
\begin{claim}
The partition $P$ is balanced; that is, all of the parts are of size 16.
\end{claim}
\begin{proof}
From the previous claim, $|Q_2 \setminus P_1| \le 10$ and we already know that $|Q_1 \setminus P_1| \geq 10$, so $|Q_1 \setminus P_1| = |Q_2 \setminus P_1| = 10$. This and Lemma~\ref{lemma-cube-or-minusPoint} imply that the partition $P$ is balanced, since the smallest part $P_1$ has to have size at least $16$ in order for $P_1 \cup Q_1$ to contain at least $26$ points.
\end{proof}
We now know $|P_1| = 16$. Then, since $|P_1 \cup Q_1| = |P_1 \cup Q_2| = 26$, $P_1$ is a $16$-set such that adding two disjoint $10$-sets, $Q_1 \setminus P_1$ and $Q_2 \setminus P_1$, creates two distinct $26$-sets contained inside $3 \times 3 \times 3$ cubes.
\begin{claim}
The set $P_1$ consists of $16$ points in a $2 \times 3 \times 3$ prism.
\end{claim}
\begin{proof}
From the previous claims, $P_1 \cup Q_1 = C_1 \setminus \{p_1\}$ and $P_1 \cup Q_2 = C_2 \setminus \{p_2\}$ for some $3 \times 3 \times 3$ cubes $C_1$ and $C_2$ and points $p_1$ and $p_2$. These two sets intersect precisely at $P_1$, implying that $C_1 \neq C_2$. On the other hand, two distinct cubes intersect in either a $2 \times 2 \times 2$ cube, a $2 \times 2 \times 3$ prism, or a $2 \times 3 \times 3$ prism. Since $P_1$ is a $16$-set that lies inside this intersection, only the last option is large enough and $P_1$ consists of $16$ points in a $2\times 3 \times 3$ prism.
\end{proof}
Finally, since the partition is balanced, the entire argument is symmetric and we see that every $P_i \in P$ consists of $16$ points in a $2\times 3\times 3$ prism. By Lemma~\ref{lemma-four-prisms-impossible}, four sets of this structure cannot partition $[4]^3$. Thus the sets $P_i$ are not a partition of $[4]^3$, which is the desired contradiction.
\end{proof}

\section{Windmills} \label{sec-windmills}
In this section, we determine the hat-guessing number of most windmill graphs, defined below. In particular, we prove Theorems~\ref{thm-wd-2k-2} and \ref{thm-wd-dn}.
\begin{definition}
    The \emph{windmill graph} $W_{k,n}$ is the graph on $(k-1)n + 1$ vertices obtained by gluing $n$ copies of $K_k$ together at a single vertex. We call the single distinguished vertex the {\it axle} of $W_{k,n}$ and each of the $n$ disjoint copies of $K_{k-1}$ not containing the axle a {\it blade} of $W_{k,n}$.
\end{definition}
Our plan of attack is to reduce hat-guessing on $W_{k,n}$ to modified hat-guessing problems on the individual blades, where the set of possible hat assignments is restricted to a prescribed subset of $[q]^{k-1}$. The following definition appears in~\cite{alon2018hat}.
\begin{definition} \label{def-partially-q}
    If $G$ is a graph on $n$ vertices, we say that a $S\subseteq [q]^n$ of possible hat assignments is a {\it solvable set of $G$ with $q$ colors} if the hat-guessing game on $G$ with $q$ colors can be won with the additional information that the hat assignment is in $S$. If $G$ and $q$ are clear from context, we simply say that $S$ is a solvable set.
\end{definition}
Thus $S=[q]^n$ is a solvable set of $G$ if and only if $\HG(G) \geq q$. In this section we will be primarily concerned with solvable sets of complete graphs.

\begin{lemma} \label{lem-H}
    If $q \ge n \ge 1$, then the size of the largest solvable set of $K_n$ with $q$ colors is $nq^{n-1}$.
\end{lemma}
\begin{proof}
    First, we show that the size of the largest solvable set of $K_n$ with $q$ colors is at most $nq^{n-1}$. Let $S_i$ be the set of hat assignments in which player $i$ guesses correctly, $x_i$ be the hat color of the $i$th player, and $f_i$ be the guessing function of the $i$th player. Then for all $(x_1,\dots,x_n) \in S_i$, $x_i = f_i(x_1,\dots,\hat{x_i},\dots,x_n)$.\footnote{We use $\hat{x_i}$ to mean that $x_i$ is omitted from the list.} Then $|S_i| \le q^{n-1}$ because $x_i$ is determined by the other $n-1$ hat colors. So the size of the largest solvable set of $K_n$ with $q$ colors is at most $\sum_{i=1}^n |S_i| \le nq^{n-1}$.
    
    Next, we show that the size of the largest solvable set of $K_n$ with $q$ colors is at least $nq^{n-1}$. Let $S$ be all of those hat assignments for which the sum of the hat colors is between $0$ and $n-1 \pmod q$. If we index the players from $0$ to $n-1$, then each player guesses that its hat color is the one that will make the sum of all their hat colors equal to its index modulo $q$. Thus $S$ is a partially $q$-solvable set for $K_n$ with $nq^{n-1}$ elements.
\end{proof}

Next, we need a simple lemma that reduces hat-guessing on windmills to packing certain solvable sets on disjoint unions of cliques. It is the analog of Lemma~\ref{lemma-cube-condition} for windmills.

\begin{lemma}\label{lem-windmill-condition}
    We have $\HG(W_{k,n})\ge q$ if and only if there exists $q$ disjoint sets in $[q]^{(k-1)n}$ of the form
    \[
    \overline{S_1} \times \overline{S_2} \times \cdots \times \overline{S_n}
    \]
    where $S_i\subseteq [q]^{k-1}$ is a solvable set of $K_{k-1}$.
\end{lemma}
\begin{proof}
We show how to give a winning guessing strategy for $W_{k,n}$ given such a collection of $q$ sets. If these sets are $P_1, P_2, \ldots , P_q \subseteq [q]^{(k-1)n}$, arbitrarily expand these sets to a partition $Q_1 \sqcup \cdots \sqcup Q_q = [q]^{(k-1)n}$ where $P_i \subseteq Q_i$. If $v$ is the axle of $W_{k,n}$, this is a partition of the possible colorings of $W_{k,n}\setminus \{v\}$. Let $v$ guess color $i$ if it sees that a coloring of $W_{k,n} \setminus \{v\}$ that lies in $Q_i$.

It remains to give the guessing strategies for the other vertices of $W_{k,n}$. Suppose $P_i = \overline{S_{i,1}} \times \overline{S_{i,2}} \times \cdots \times \overline{S_{i,n}}$ where $S_{i,j}\subseteq [q]^{k-1}$ is a solvable set of the $j$-th blade, which is a copy of $K_{k-1}$. 

Fix some $1\le j \le n$ and consider the vertices of the $j$-th blade of $W_{k,n}$. Let $c_j \in [q]^{k-1}$ be the vector of hat assignments on these vertices. Since $S_{i,j}$ is a solvable set of $K_{k-1}$, it follows that there exists some guessing strategy $g_{i,j}$ on blade $j$ which guarantees that some player guesses correctly if $c_j \in S_{i,j}$. 

The guessing strategy for vertices of blade $j$ will be as follows. Each non-axle vertex first records the color $i$ as the color of the axle vertex $v$. Then, the vertices of blade $j$ restrict their attention to the other vertices of the same blade and follow guessing strategy $g_{i,j}$ so as to guarantee a win if $c_j \in S_{i,j}$.

We see that as long as $c_j \in S_{i,j}$ for some $j$, some vertex in blade $j$ guesses correctly and the players win. On the other hand, if $c_j \not \in S_{i,j}$ for any $j$, then that means that $c_1 \times c_2 \times \cdots \times c_n \in P_i\subseteq Q_i$, and so the axle $v$ guesses color $i$ and thus guesses correctly. In any case, we have a winning strategy and this completes the proof of the if direction. The only if direction is similar.
\end{proof}

The final ingredient for Theorem~\ref{thm-wd-2k-2} is the existence of a specific type of solvable set for $K_{k-1}$ with $2k-2$ colors.

\begin{lemma} \label{c-comp}
    For all $k\ge 2$, there is a set $C\subseteq[2k-2]^{k-1}$ such that both $C$ and $\overline{C}$ are solvable for $K_{k-1}$ with $2k-2$ colors.
\end{lemma}
\begin{proof}
    The set $C$ is to be a subset of $[2k-2]^{k-1}$. For a binary vector $v \in \{0,1\}^{k-1}$ define the subhypercube $C_v \subset [2k-2]^{k-1}$ to be 
    \[
    C_v \coloneqq \{x\in [2k-2]^{k-1} \mid (k-1)v_i \leq x_i < (k-1)(v_i + 1) \text{ for all } i\le k-1\}. 
    \]
    Thus, the sets $C_v$ partition $[2k-2]^{k-1}$ into $2^{k-1}$ hypercubes with side length $k-1$. Let $C$ be the union of those $C_v$ for which $v\in \{0,1\}^{k-1}$ has an odd number of $1$s.
    
    We now check that $C$ is solvable. Indeed, in the hat-guessing game on $K_{k-1}$, each vertex $u_i$ sees the colors on all the other vertices, and can determine from this information exactly two possible hypercubes $C_v$, $C_{v'}$ in which the hat assignment vector must lie. The binary vectors $v, v'$ will differ in exactly coordinate $i$ because $u_i$ has no information about its own color. However, with the additional constraint that the hat assignment is in $C$, exactly one of these vectors $v, v'$ will have an odd number of $1$'s. In other words, every vertex in $K_{k-1}$ will be able to determine the (same) hypercube $C_v$ with side length $k-1$ in which the hat assignment vector lies.
    
    From this point, the game is reduced to the $(k-1)$-color hat-guessing game on $K_{k-1}$, which we know to be a guaranteed win for the players. This proves that $C$ is solvable, and the proof for $\overline{C}$ is analogous.
\end{proof}

We are ready to prove Theorem~\ref{thm-wd-2k-2}, which states that for $k\ge 2$ and $n \ge \log_2(2k-2)$,
\[\HG(W_{k,n}) = 2k-2.\]

\begin{proof}[Proof of Theorem~\ref{thm-wd-2k-2}.]
We first show the upper bound. Suppose $\HG(W_{k,n}) = q \geq 2k-1$ for the sake of contradiction. By Lemma~\ref{lem-windmill-condition}, there exists $q\ge 2$ disjoint subsets of $[q]^{(k-1)n}$ which are products of complements of solvable sets of $K_{k-1}$.
    
    Consider two of these subsets, $\overline{S_1} \times \dots \times \overline{S_n}$ and $\overline{T_1} \times \dots \times \overline{T_n}$, where $S_i$ and $T_i$ are solvable sets of $K_{k-1}$. By Lemma~\ref{lem-H}, the size of the largest solvable set of $K_{k-1}$ with $q$ colors is $(k-1)q^{k-2}$. Thus, $|S_i| \le (k-1)q^{k-2} < \frac{1}{2} q^{k-1}$ if $q\ge 2k-1$, and similarly for $T_i$. This implies $\overline{S_i} \cap \overline{T_i} \ne \emptyset$ for all $i$. This contradicts the assumption that the products $\overline{S_1} \times \dots \times \overline{S_n}$ and $\overline{T_1} \times \dots \times \overline{T_n}$ were disjoint, and we are done.
    
    Now we show the lower bound. By Lemma~\ref{lem-windmill-condition}, it suffices to exhibit $q=2k-2$ disjoint sets in $[q]^{(k-1)n}$ that are products of complements of solvable sets of $K_{k-1}$. By Lemma~\ref{c-comp}, there exists a set $C$ such that both $C$ and $\overline{C}$ are solvable sets of $K_{k-1}$ with $q=2k-2$ colors. For convenience, let $C_0 = C$ and $C_1 = \overline{C}$.
    
    For each $0\le x \le q-1$, define $x_i$ to be the $i$-th (least significant) digit of $x$ in binary, and
    \[
    P_x \coloneqq C_{x_1} \times C_{x_2} \times \dots \times C_{x_n}.
    \]
    Note that since $n\ge \log_2 (2k-2)$, we get $q-1 =2k-3 \le 2^n - 1$, and so all of the sets $P_x$ above are disjoint. Since each $C_{x_i}$ is a solvable set of $K_{k-1}$, this completes the proof.
    \end{proof}
To prove Theorem~\ref{thm-wd-dn}, we will construct certain guessing strategies using additive combinatorics.
\begin{definition}\label{def-1-int-under-trans}
    We say that a collection $\{A_1, \dots, A_n\}$ of sets $A_i\subseteq \mathbb{Z}/m\mathbb{Z}$ is {\it difference-disjoint} if $\bigcap_{i=1}^{n}\left( A_i - A_i\right) = \{ 0 \}$. 
\end{definition}

An equivalent definition is that $\{A_1, \dots, A_n\}$ is difference-disjoint if and only if for all $(c_1, \dots, c_n) \in (\mathbb{Z}/m\mathbb{Z})^n$, we have $|\bigcap_{i=1}^{n} \left(A_i + c_i\right)| \le 1$. Indeed, this latter intersection contains at least two elements if and only if there is a pair of elements in each $A_i$ with the same nonzero difference. We proceed by constructing certain optimal difference-disjoint collections.
\begin{lemma} \label{lem-resid}
    For all $n\ge 1$ and $d\ge 2$, there exists a difference-disjoint collection $\{A_1,\ldots, A_n\}$ of $n$ sets $A_i \subseteq \mathbb{Z}/d^n\mathbb{Z}$ with $|A_i| = d^{n-1}$ for $i=1,\ldots, n$.
\end{lemma}
\begin{proof}
Let $\digit{d}{x}{i}$ be the $i$-th digit of $x$ in base $d$, where $\digit{d}{x}{0}$ is the least significant digit; that is $\digit{d}{x}{i} = \left\lfloor{\frac{x}{d^i}}\right\rfloor - d\left\lfloor{\frac{x}{d^{i+1}}}\right\rfloor$. Let $\ruler{d}{x}$ be the largest power of $d$ that divides $x$, or equivalently the number of trailing zeros in the base-$d$ representation of $x$.

With this notation, let $A_i = \{ x \mid \digit{d}{x}{i-1} = 0\}$ for $i = 1,\ldots, n$. We claim that for each nonzero $a \in \mathbb{Z}/d^n\mathbb{Z}$, $a \notin \left(A_{\ruler{d}{a}+1} - A_{\ruler{d}{a}+1}\right)$. Indeed, suppose $a = x-y$ for $x, y \in A_{\ruler{d}{a}+1}$, so that $\digit{d}{x}{\ruler{d}{a}} = \digit{d}{y}{\ruler{d}{a}} = 0$. In this situation, it is impossible for $\digit{d}{x-y}{\ruler{d}{a}} = 0$ unless $\digit{d}{x-y}{t} \neq 0$ for some $t < \ruler{d}{a}$. But this would imply $\ruler{d}{x-y} < \ruler{d}{a}$, which is a contradiction. Thus, the sets $A_i$ form a difference-disjoint collection as desired.
\end{proof}

We can now prove Theorem~\ref{thm-wd-dn}. Recall the statement: for $n\ge 1$ and $d\ge 2$, 
\[
\HG(W_{d^n-d^{n-1}+1,n}) = d^n.
\]
Our construction is a generalization of a strategy suggested by Alweiss~\cite{alweiss} for $W_{3,2}$.

\begin{proof}[Proof of Theorem~\ref{thm-wd-dn}]
    Let $k = d^n - d^{n-1} + 1$ and $q=d^n$. We prove separately that $\HG(W_{k,n}) \le q$ and $\HG(W_{k,n}) \ge q$.

    First we show $\HG(W_{k,n}) \le q$. Suppose instead that $\HG(W_{k,n}) \ge q+1$. By Lemma~\ref{lem-windmill-condition} there must exist $q+1$ disjoint sets $P_1,\ldots, P_{q+1}$ in $[q+1]^{(k-1)n}$, where each $P_j$ is of the form
    \[
    \overline{S_1} \times \overline{S_2} \times \cdots \times \overline{S_n}
    \]
    where $S_i$ is a solvable set of $K_{k-1}$ with $q+1$ colors. By Lemma~\ref{lem-H}, any such solvable set has size at most $(k-1)(q+1)^{k-2}$, and so $|S_i| \le (k-1)(q+1)^{k-2}$. Hence
    \[
    |P_j| = \prod_{i=1}^{n} |\overline{S_i}| \ge ((q-k)(q+1)^{k-2})^n.
    \]
    We claim that this is impossible because the sets $P_j$ are simply too large to all fit inside $[q+1]^{(k-1)n}$. Indeed,
    \[
    (q+1)|P_j| \ge (q+1)^{(k-2)n+1}\cdot (q-k)^n = (q+1)^{(k-2)n + 1} \cdot (d^{n-1}+1)^n > (q+1)^{(k-1)n}.
    \]
    In the last line, we used the inequality $(d^{n-1}+1)^n > (d^n + 1)^{n-1}$ which holds for all $d\ge 2$ and $n\ge 1$. This completes the proof that $\HG(W_{k,n}) \le q$.
    
    We finish by showing $\HG(W_{k,n}) \ge q$. Identify the set of colors $[q]$ with the elements of $\mathbb{Z}/q\mathbb{Z}$, and let $\{A_1,\ldots, A_n\}$ be the difference-disjoint collection in $\mathbb{Z}/q\mathbb{Z}$ constructed by Lemma~\ref{lem-resid}, so that $|A_i| = d^{n-1}$ for each $i$. For any set of residues $A\subseteq \mathbb{Z}/q\mathbb{Z}$, define $S(A)$ to be the set
    \[
    S(A)\coloneqq \{ (x_1,\ldots, x_{k-1}) \in [q]^{k-1} \mid x_1 + \cdots + x_{k-1} \not \in A\}
    \]
    of all hat assignments to $K_{k-1}$ whose sum is not in $A$.
    
    Our first claim is that for any set $A$ with $|A| = d^{n-1}$, $S(A)$ is a solvable set of $K_{k-1}$. Indeed, $S(A)$ consists of all hat assignments with sum in $\overline{A}$, which has size $q - |A| = d^n - d^{n-1} = k-1$. This set is solvable because we can assign each element $b_i \in \overline{A}$ to a distinct vertex $i$ of $K_{k-1}$ and have vertex $i$ guess the color that would make the total sum of the hat colors $b_i$.
    
    Now, define $q=d^n$ sets $P_0,\ldots, P_{q-1}$ by
    \[
    P_j \coloneqq \overline{S(A_1 + j)} \times \overline{S(A_2 + j)} \times \cdots \times \overline{S(A_n + j)}.
    \]
    Here $A_i+j$ denotes the$\pmod q$-translation of the set $A_i$ by $j$. Since $S(A_i + j)$ is solvable in $K_{k-1}$, it remains to show that the sets $P_j$ are all disjoint in order to apply Lemma~\ref{lem-windmill-condition}. If not, there would exist distinct $j,j' \in [q]$ such that $(x_1,\ldots, x_n) \in P_j \cap P_{j'}$ for some vector $(x_1,\ldots,x_n) \in [q]^{(k-1)n}$. Equivalently, writing $\sigma(x_i)$ for the sum of the coordinates of $x_i$, this means that $\sigma(x_i) \in (A_i +j) \cap (A_i + j')$ for every $i$. But then $j' - j \in A_i - A_i$ for every $i$, which contradicts the fact that the sets $A_i$ form a difference-disjoint collection.
    
    This completes the construction of $q$ disjoint sets $P_0,\ldots, P_{q-1}$ satisfying the conditions of Lemma~\ref{lem-windmill-condition}, and proves that $\HG(W_{k,n}) \ge q$ as desired.
\end{proof}

\section{Concluding Remarks} \label{sec-concl}
Gadouleau and Georgiou proved that $\HG\left(K_{n, n}\right) \leq n + 1$~\cite{gadgeor}. This is tight for $n = 1$ and $n = 2$. However, in this paper, we proved that $\HG\left(K_{3, 3}\right) = 3$. It remains an interesting open question to determine the value of $\HG\left(K_{n, n}\right)$ for $n > 3$. We conjecture the following generalization of Theorem~\ref{thm-k33}.
\begin{conjecture}
For $n \geq 3$, $\HG\left(K_{n,n}\right) \leq n$.
\end{conjecture}
The windmill graph $W_{4,3}$ has hat-guessing number 6, disproving the conjecture that all planar graphs have hat-guessing number at most $4$ from~\cite{bosek2019hat}. He and Li~\cite{he2020hat} previously gave another planar graph with a hat-guessing number of 6, namely $B_{2,n}$ for sufficiently large $n$. Recently, \cite{alonnew} constructed a planar graph with a hat-guessing number of 12. It remains open whether the hat-guessing number of planar graphs is bounded.
\begin{question}
Do there exist planar graphs with arbitrarily large hat-guessing number?
\end{question}
Since planar graphs have a Hadwiger number (largest clique minor) of at most 4, a more general question is whether the hat-guessing number is upper bounded by some function of the Hadwiger number.
\begin{question}
Is there a function $f$ such that $\HG(G) \le f(h(G))$, where $h(G)$ is the Hadwiger number of $G$?
\end{question}
All of our results support the following conjecture about the upper bound of all graphs in terms of the maximum degree $\Delta$. This conjecture, first proposed in \cite{alon2018hat}, tightens the folklore upper bound of $e\Delta$ given by the Lov\'asz Local Lemma.
\begin{conjecture}
$\HG(G) \leq \Delta + 1$.
\end{conjecture}
Books and windmills are both generalizations of the complete graph. Books glue multiple copies of the complete graph together by leaving one vertex unique to each copy. In contrast, windmills glue multiple copies of the complete graph together at exactly one vertex. The case of gluing multiple copies of the complete graph at an intermediate number of vertices remains unexplored.

Perhaps the most interesting question in hat guessing is whether far-apart vertices can coordinate their guesses in a way that contributes to the hat-guessing number. Almost all graphs studied to date\textemdash including books, windmills, and the complete bipartite graph\textemdash have a diameter of at most 2. Define a graph to be {\it hat-minimal} if every proper subgraph has a smaller hat-guessing number.
\begin{question}
Do there exist hat-minimal graphs with arbitrarily large diameter and hat-guessing number?
\end{question}
Another related question is whether graphs with high girth can have high hat-guessing number.
\begin{question}
Do there exist graphs with arbitrarily large girth and hat-guessing number?
\end{question}
It seems that the fundamental roadblock to answering either question is the absence of guessing strategies in which far-away vertices can coordinate effectively. The only graphs with higher diameter or girth for which anything interesting is known are cycles, for which the hat guessing number is at most 3~\cite{szczechla}, and graphs with more than one cycle, for which the hat guessing number is at least 3~\cite{KL}. It would already be interesting to find hat-minimal graphs with hat-guessing number $4$ and arbitrarily large girth or diameter.

\section{Acknowledgments}
We would like to thank Pawel Grzegrzolka and the Stanford Undergraduate Research Institute in Mathematics, at which this research was conducted. We are also grateful to Ryan Alweiss for an idea that simplified our exposition on windmill graphs, and to Noga Alon, Jacob Fox, Jarek Grytczuk, Zhuoer Gu, Ben Gunby, and Ray Li for many stimulating conversations.

\section{Appendix}
Here we prove Lemmas~\ref{lemma-cube-or-minusPoint} and \ref{lemma-three-cubes}. In both proofs it will be helpful to study sets which are complements of $3\times 3\times 3$ cubes.

\begin{definition} If $p\in [4]^3$, the \emph{big Hamming ball about $p$}, denoted $\bHb(p)$, is the set of all points sharing at least one coordinate with $p$. 
\end{definition}

Note that $\bHb(p)$ is always the complement of a $3 \times 3 \times 3$ cube. We first prove Lemma~\ref{lemma-cube-or-minusPoint}, which states that if a two-intersection of four $3 \times 3 \times 3$ cubes contains at most $29$ points, then it must either be a $3 \times 3 \times 3$ cube or a $3 \times 3 \times 3$ cube missing a point.

The relationship of two points can be captured by the following notion.
\begin{definition}
The \emph{Hamming distance} of two points $p$ and $q$, notated $d(p,q)$, is the number of coordinates that $p$ and $q$ disagree on.
\end{definition}
We can then see that for $p \in [4]^3$, $\bHb(p)$ is the set of $q$ such that $d(p,q) < 3$.

\begin{proof}[Proof of Lemma~\ref{lemma-cube-or-minusPoint}]
    Suppose we have four $3 \times 3 \times 3$ cubes $C_1 = \overline{\bHb(p_1)}$, $C_2 = \overline{\bHb(p_2)}$, $C_3 = \overline{\bHb(p_3)}$, and $C_4 = \overline{\bHb(p_4)}$. We condition on the set of four points $\{p_1, p_2, p_3, p_4\}$. We can assume without loss of generality that $p_1 = (0,0,0)$. Further, if $p_i = p_j$ with $i \neq j$, then the two-intersection is either a $3 \times 3 \times 3$ cube or contains at least $31$ points (a $3 \times 3 \times 3$ cube with at least $4$ extra two-intersections), which is in accordance with our lemma. We now assume that no two of the four points are equal.
    
    First, suppose two of the four points have Hamming distance 1. Without loss of generality, this occurs with $p_2 = (1,0,0)$. If $d(p_1, p_2) = d(p_1, p_3) = d(p_1, p_4) = 1$, then the two-intersection is a $3 \times 3 \times 3$ cube missing a point, unless there are three points in an axis-parallel line, in which case we already have a $4 \times 4 \times 3$ prism, and we are done. We can then assume that $p_3$ has at least $2$ nonzero coordinates. So without loss of generality, $p_3 \in \{(0,1,1), (1,1,0), (2,1,0), (2,1,1)\}$. (These cases are sufficient because with the existing choices of $p_1$ and $p_2$, only the $x$-coordinate is distinct while the $y$- and $z$-coordinates are symmetric, and the values $1$ and $2$ are symmetric with respect to the $y$- and $z$-coordinates.) We consider the cases in turn. (Several of these cases require exhaustively considering the possibilities for $p_4$, although symmetries make the task easier.)
    
    First, if $p_3 = (0,1,1)$, then we can create a two-intersection of size $30$ by setting $p_4 = (0,1,0)$, but one can see that no other choice of $p_4$ achieves a smaller two-intersection.
    
    Second, if $p_3 = (1,1,0)$, then it is possible to achieve a two-intersection of size $26$ by setting $p_4 = (1,0,1)$, and this is the $3 \times 3 \times 3$ cube missing a point as desired. Further, we can get a two-intersection of $30$ by setting $p_4 = (0,0,1)$ or $p_4 = (1,1,1)$, but no smaller two-intersections are possible besides the $3 \times 3 \times 3$ cube missing a point.
    
    Third, if $p_3 = (2,1,0)$, this case can easily be dismissed because the two-intersection is already size $30$.
    
    Fourth, if $p_3 = (2,1,1)$, we can do a two-intersection of size $32$ with $p_4 = (0,1,0)$, but no smaller two-intersection is possible, which is not difficult to see once one notes that the size of the two-intersection is already size $26$ after adding the third $3 \times 3 \times 3$ cube, and the last $3 \times 3 \times 3$ cube must add at least 6 more.
    
    Thus, we have proven the lemma if two points have a Hamming distance of $1$ between them. We now suppose that all pairs of points have a Hamming distance of at least $2$ between them. For $i = 1, \dots, 4$, let $p_i = (x_i, y_i, z_i)$. Then, $C_i = S_i \times \overline{\{z_i\}}$, where $S_i$ is the combinatorial square $\overline{\{x_i\}} \times \overline{\{y_i\}}$. Because any two points $p_i$ and $p_j$ have a Hamming distance of at least 2, we see that they must disagree in at least one of the $x$- or $y$-coordinates, so all of the $S_i$ must be distinct.
    
    One can verify that two $3 \times 3$ squares in $[4]^2$ must two-intersect in at least $4$ points, and three $3 \times 3$ squares must two-intersect in at least $8$ points. Although not as obvious, one can also verify that four distinct $3 \times 3$ squares must two-intersect in at least $12$ points (the minimal two-intersection is achieved with the squares defined by the points $(0,0)$, $(1,0)$, $(0,1)$, $(1,1)$). This implies that given two cubes $C_i = S_i \times \overline{\{z_i\}}$ and $C_j= S_j \times \overline{\{z_j\}}$, the cubes must two-intersect in at least 4 points in the $4 \times 4$ cross section of $[4]^3$ given by a fixed $z$-coordinate in $\overline{\{z_i\}} \cap \overline{\{z_j\}}$. Similarly, given three cubes $C_i = S_i \times \overline{\{z_i\}}$, $C_j= S_j \times \overline{\{z_j\}}$, $C_k= S_k \times \overline{\{z_k\}}$, the cubes must two-intersect in at least 8 points in a cross section given by a fixed $z \in \overline{\{z_i\}} \cap \overline{\{z_j\}} \cap \overline{\{z_k\}}$. Finally, given four cubes $C_i = S_i \times \overline{\{z_i\}}$, $C_j= S_j \times \overline{\{z_j\}}$, $C_k= S_k \times \overline{\{z_k\}}$, and $C_l = S_l \times \overline{\{z_l\}}$, where all of $S_i, S_j, S_k$, and $S_l$ are distinct, the four cubes must two-intersect in at least 12 points in a cross-section given by a fixed $z \in \overline{\{z_i\}} \cap \overline{\{z_j\}} \cap \overline{\{z_k\}} \cap \overline{\{z_l\}}$. 
    
    There are five ways to choose the $z_i$ without loss of generality. 
    \[(z_1, z_2, z_3, z_4) \in \{(0, 0, 0, 0), (1, 0, 0, 0), (1, 1, 0, 0), (2, 1, 0, 0), (3, 2, 1, 0)\}\]  
    We will rewrite this information in the following way: let $n_i$ be the number of sets among $\overline{\{z_1\}}, \overline{\{z_2\}}, \overline{\{z_3\}}, \overline{\{z_4\}}$ that contain $i$. Then, 
    \[(n_1, n_2, n_3, n_4) \in \{(0,4,4,4), (1,3,4,4), (2,2,4,4), (2,3,3,4), (3,3,3,3)\}\]
    Using the facts from the previous paragraph, this gives us a lower bound for the total number of two-intersection points for each of these arrangements. For instance, in the arrangement $(0, 4, 4, 4)$, the four distinct $3 \times 3 \times 3$ cubes $C_1$, $C_2$, $C_3$, and $C_4$ all share the $z$-coordinates $1$, $2$, and $3$. Since these four $3 \times 3 \times 3$ cubes must two-intersect in at least 12 points for a fixed $z$-coordinate, there are at least $3 \cdot 12 = 36$ two-intersections total. The other cases can be analyzed similarly; the results are summarized below.
    
    \begin{center}
    \begin{tabular}{ |c|c| } 
    \hline
    Arrangement & Number of two-intersections \\
    \hline
    $(0,4,4,4)$ & $\ge 3\cdot 12 = 36$ \\
    \hline
    $(1,3,4,4)$ & $\ge 2\cdot 12 + 8 = 32$ \\
    \hline
    $(2,2,4,4)$ & $\ge 2\cdot 12 + 2\cdot 4 = 32$ \\
    \hline
    $(2,3,3,4)$ & $\ge 12 + 2\cdot 8 + 4 = 32$ \\
    \hline
    $(3,3,3,3)$ & $\ge 4\cdot 8 = 32$ \\
    \hline
    \end{tabular}
    \end{center}
    
    Since there are at least $30$ two-intersection points in all of these cases, none of them can be a counterexample to the lemma, so the lemma is proved.
\end{proof}

We now prove Lemma~\ref{lemma-three-cubes}, which states that three $3 \times 3 \times 3$ cubes in $[4]^3$ must two-intersect at a minimum of 20 points.
\begin{proof}[Proof of Lemma~\ref{lemma-three-cubes}]
Call the $3 \times 3 \times 3$ cubes $C_1 = \overline{\bHb(p_1)}$, $C_2 = \overline{\bHb(p_2)}$, and $C_3 = \overline{\bHb(p_3)}$, where $p_1, p_2, p_3 \in [4]^3$. Let $x_1$, $x_2$, and $x_3$ represent the number of distinct $x_1$-, $x_2$-, and $x_3$-coordinates among $p_1$, $p_2$, and $p_3$, respectively. Let $I$ be the two-intersection of $C_1, C_2,$ and $C_3$. By inclusion-exclusion,
\[ \lvert I \rvert = \lvert C_1 \cap C_2\rvert + \lvert C_2 \cap C_3 \rvert  + \lvert C_1 \cap C_3 \rvert - 2 \lvert C_1 \cap C_2 \cap C_3 \rvert \]
We see that a point is in $C_1 \cap C_2 \cap C_3$ if and only if none of its coordinates is used by $p_1$, $p_2$, or $p_3$. Thus $\lvert C_1 \cap C_2 \cap C_3 \rvert = (4-x_1)(4-x_2)(4-x_3)$.

Let $d_1 = d(p_1, p_2)$, $d_2 = d(p_2, p_3)$, and $d_3 = d(p_1, p_3)$ be the Hamming distances between each pair of points. Then, in an $x_1$-, $x_2$-, or $x_3$-coordinate where $p_1$ and $p_2$ agree, points in $C_1 \cap C_2$ take on the three values which are not the one $p_1$ and $p_2$ use. In a coordinate where $p_1$ and $p_2$ disagree, points in $C_1 \cap C_2$ take on the two values which are used by neither $p_1$ nor $p_2$. Thus $\lvert C_1 \cap C_2 \rvert = 3^{3-d_1} 2^{d_1}$. By the same argument, $\lvert C_2 \cap C_3 \rvert = 3^{3 - d_2} 2^{d_2}$ and $\lvert C_1 \cap C_3 \rvert = 3^{3 - d_3} 2^{d_3}$. Then,
\begin{align*}
    \lvert I \rvert &= \sum_{i = 1}^3 3^{3 - d_i} 2^{d_i} - 2(4-x_1)(4-x_2)(4-x_3) \\
    &= 27 \sum_{i = 1}^3 \left(\frac{2}{3}\right)^{d_i} - 2(4-x_1)(4-x_2)(4-x_3) \\
    &\geq 27 \cdot 3 \cdot \left(\frac{2}{3}\right)^{(d_1 + d_2 + d_3) / 3} - 2(4-x_1)(4-x_2)(4-x_3),
\end{align*}
where the last line above follows from the arithmetic mean\textendash geometric mean inequality. We now perform a change of variables so $a_j = |\{i: x_i = j\}|$ for $j \in \{1, 2, 3\}$. Then, we see that $d_1 + d_2 + d_3 = 2a_2 + 3a_3$ and $(4-x_1)(4-x_2)(4-x_3) = 3^{a_1} 2^{a_2}$. Substituting gives
\[    \lvert I \rvert \geq 27 \cdot 3 \cdot \left(\frac{2}{3}\right)^{(2a_2 + 3a_3) / 3} - 3^{a_1} 2^{a_2 + 1}.\]
Conditioned on $a_j \in \Z$ and $0 \leq a_j \leq 3$ for $j \in \{1,2,3\}$, with $a_1 + a_2 + a_3 = 3$, calculating all cases shows the minimum of the right hand side is 20, attained when $a_1 = 0$, $a_2 = 3$, and $a_3 = 0$.
\end{proof}
\end{document}